\documentclass{amsart}
\usepackage{amsfonts}
\usepackage{amsmath}
\usepackage{amssymb}
\usepackage{graphicx}
\newtheorem{theorem}{Theorem}
\theoremstyle{plain}

\newtheorem*{assumption}{Assumptions}

\newtheorem{lemma}{Lemma}

\newtheorem{proposition}{Proposition}

\theoremstyle{definition}

\newtheorem{definition}{Definition}

\newtheorem{remark}{Remark}
\theoremstyle{remark}

\numberwithin{equation}{section}
\title{Stability of the Ricci Yang-Mills flow at Einstein Yang-Mills metrics}
\author{Andrea Young}

\begin{document}

\begin{abstract}
Let $P$ be a principal $U(1)$-bundle over a closed manifold $M$.   On $P$, one can define a modified version of the Ricci flow called the Ricci Yang-Mills flow, due to these equations being a coupling of Ricci flow and the Yang-Mills heat flow.  We use maximal regularity theory and ideas of Simonett concerning the asymptotic behavior of abstract quasilinear parabolic partial differential equations to study the stability of the volume-normalized Ricci Yang-Mills flow at Einstein Yang-Mills metrics in dimension two.  In certain cases, we show the presence of a center manifold of fixed points, while in others, we show the existence of an asymptotically stable fixed point.
\end{abstract}
\maketitle

By writing the Ricci flow equations for a metric on a $U(1)$-bundle in the Kaluza-Klein ansatz with fixed fiber size, one obtains a coupled system of equations--an equation that resembles the Ricci flow on the base metric and an equation that resembles the Yang-Mills heat flow on the connection 1-form.  We call this coupled system of equations the Ricci Yang-Mills flow.  Recall that the Yang-Mills heat flow is well-behaved in low dimensions, while the Ricci flow can become singular even in dimension two.  Thus, one hopes to exploit the nice behavior of the Yang-Mills heat flow part of the system to obtain convergence results of the Ricci Yang-Mills flow.

In 1982, Richard Hamilton \cite{Ha1} proposed the Ricci flow as a means to study 3-manifolds with positive Ricci curvature.  Specifically, let $(M^n,g)$ be an $n$-dimensional Riemannian manifold with metric $g$.  The Ricci flow equations are defined to be
\begin{equation}
\label{RF}
\frac{\partial g}{\partial t}=-2Rc,
\end{equation}
where $Rc$ is the Ricci curvature of $g$.  It is well known that the Ricci flow is a weakly parabolic system of equations.  In his seminal paper, Hamilton showed that a closed 3-manifold with positive Ricci curvature is diffeomorphic to a spherical space form, thus that it admits a metric of constant sectional curvature.  

Hamilton developed a program intending to use Ricci flow to prove Thurston's geometrization conjecture, which states that every closed 3-manifold admits a geometric decomposition.  Hamilton's papers yielded great progress towards this goal, studying a multitude of topics such as singularity formulation \cite{Ha_sing}, compactness theorems \cite{Ha_compact}, and nonsingular solutions \cite{Ha_nonsing}.  The recent work of Grisha Perelman, which in particular combined comparison geometry and partial differential equations, has provided much progress in the direction of studying geometrization \cite{Pe1}, \cite{Pe2}.   Additionally, Ricci flow has proven to be a very fruitful area of study in its own right (for overviews, see, e.g., \cite{ChKn}, \cite{Ch}).

On the other hand, the Yang-Mills heat flow is a gauge-theoretic heat equation; that is, it is a differential equation for a field on a principal fiber bundle.  Let $P$ be a smooth principal $G$-bundle over a smooth closed manifold $M$.  If $A$ is a connection on $P$, then $A$ yields an exterior covariant derivative, denoted $D_A$, that acts on $k$-forms with values in $\mathcal{G}$, the Lie algebra of $G$.  The curvature of $A$ is then $F(A)=D_AA$.  We can define the Yang-Mills energy to be
\begin{equation}
\label{YM ftl}
\mathcal{YM}(A)=\frac{1}{2}\int_M|F(A)|^2 dV.
\end{equation} 
The $L^2$ gradient flow for this functional is the Yang-Mills heat flow:
\begin{equation}
\label{YM heat flow}
\frac{\partial A}{\partial t}=-D_A^*F(A),
\end{equation}
where $D_A^*$ is the formal adjoint of $D_A$.  The Yang-Mills heat flow was first used by Atiyah and Bott \cite{AtBo} and by Simon Donaldson \cite{Do}.  Atiyah and Bott used the Yang-Mills heat flow to study the topology of minimal Yang Mills connections.  Donaldson used it to give an analytic proof of a theorem of Narasimhan and Seshadri concerning the relation between stable holomorphic vector bundles and equivalence classes of Yang-Mills connections.  Johan Rade \cite{Ra} studied the behavior of the Yang-Mills heat flow in two and three dimensions and was able to use a technique of L. Simon to show that solutions converge as $t\to \infty$. 

In order to derive a natural coupling of the Ricci flow and the Yang-Mills heat flow, we consider the following setting.  Let $M$ be a closed Riemannian manifold with metric $\underline{g}$, and let $U\subset M$ be a local coordinate chart with coordinates  $\{x^i\}_{i=1}^n$.   Suppose $G$ is a compact Lie group with smooth metric $\bar{g}$ parametrized by the base, and $\{y^\theta\}_{\theta=n+1}^m$ are local coordinates on $G$. Then define $\pi:P\to M$ to be a principal $G$-bundle over $M$, having connection $A$.   We consider a metric $h$ on the total space $P$ of the form
\[h=\underline{g}_{ij}dx^idx^j+\bar{g}_{\theta \rho}(dy^\theta + a_k^\theta dx^k)(dy^\rho+a_l^\rho dx^l).
\]
Here, $a=\sigma^\ast A$, where $\sigma:U\to P$ is a smooth local section.  We have the following basis for one-forms: $dz^i=dx^i$ and $dz^\theta=dy^\theta+a_i^\theta dx^i$ with the corresponding frame $e_i=\frac{\partial}{\partial x^i}-a_i^\theta \frac{\partial}{\partial y^\theta}$ and $e_\theta=\frac{\partial}{\partial y^\theta}$.  
By computing the curvature quantities associated with this metric along with the additional hypothesis that the size of the fiber remains fixed, one obtains the Ricci Yang-Mills flow:
\begin{subequations}
\label{rym}
\begin{align}
\frac{\partial g_{ij}}{\partial t}&=-2R_{ij}+F^k_iF_{kj}\\
\frac{\partial a_i}{\partial t}&=-d^\star F_i.
\end{align}
\end{subequations}
Due to the sign difference that terms involving $F$ add to the Ricci tensor of the bundle, one does not expect that the Ricci Yang-Mills flow will have as its limit an Einstein metric.  Rather, the canonical metric one should hope for is an Einstein Yang-Mills metric; namely, one that is Einstein on the base and that has a Yang-Mills connection.  
\bigskip

The Ricci Yang-Mills flow has been studied simultaneously and independently by Jeffrey Streets in \cite{St} and \cite{St1}.  General convergence results in dimension $n=2$ were obtained in \cite{St}, using different methods. He used similar techniques to that of Struwe \cite{Str}, while we use the machinery of analytic semigroups to take optimal advantage of the parabolic smoothing properties of the quasilinear equation (\ref{rym}) in continuous interpolation spaces. In particular, our techniques show that a fixed point of the flow is exponentially attractive in time in the norm of a smaller  space for solutions whose initial data are close to that fixed point in a larger interpolation space.   Additionally, Dan Jane has shown that the Ricci Yang-Mills flow arises when studying magnetic flows on surfaces \cite{Ja}.

\bigskip

We would like to use maximal regularity theory to study the stability of the Ricci Yang-Mills flow at a fixed point.    In \cite{GIK}, Guenther, Isenberg, and Knopf used the general theory of \cite{Si} to study the stability of the Ricci flow at Ricci-flat metrics.  Additionally, Knopf has recently used these techniques to study the stability of locally $\mathbb{R}^N$-invariant solutions of Ricci flow \cite{Kn}.  We will consider the Ricci Yang-Mills flow to be an ODE on an infinite dimensional space.  We will linearize the right-hand side of the equation and study the spectrum of that operator.  This will determine the local behavior of the flow.  

This paper is organized as follows.   In \S1, we state the theorems concerning asymptotic behavior of general quasilinear partial differential equations; this will provide the framework for the analysis that follows.   In \S 2, we compute the linearization of the Ricci Yang-Mills flow at an Einstein Yang-Mills metric.   We recall the definition of the little H\"{o}lder spaces in \S 3, and in \S 4, we use these spaces to obtain our main stability theorems.

\bigskip

The author would like to thank her thesis advisor Karen Uhlenbeck for suggesting this flow and for much support and also Dan Knopf for many helpful conversations.

\section{Asymptotic Behavior of Quasilinear Partial Differential Equations}

Consider the general equation
\begin{subequations}
\label{quasi}
\begin{align}
\frac{\partial u}{\partial t}&=\Phi(x,t,u,Du,D^2u)\\
u(0)&=u_0.
\end{align}
\end{subequations}
For our purposes, this will be a parabolic system of partial differential equations.  Let $\bar{u}$ be a fixed point of the equation (i.e. $\Phi(\bar{u})=0$), and let  $\Sigma=\sigma(D_{\bar{u}}\Phi)\cap \mathbb{R}$.  Then

\begin{definition}
$\bar{u}$ is 
\begin{itemize}
\item\emph{linearly stable} if $\Sigma \subset (-\infty,0]$,
\item\emph{strictly linearly stable} if $\Sigma \subset (-\infty,0)$,
\item\emph{linearly unstable} if $\Sigma \cap (0,\infty)\neq \emptyset$,
\item\emph{asymptotically stable} if there exists a neighborhood about $\bar{u}$ such that every solution of the equation having initial data in that neighborhood exists for all positive time and converges to $\bar{u}$ as $t\to \infty$.
\end{itemize}
\end{definition}

In the case of the Ricci Yang-Mills flow over a compact surface, we will see that the linearized operator of the right-hand side, in the notation above, either has a zero eigenvalue or has a strictly negative spectrum.  The presence of a zero eigenvalue will correspond to the existence of a finite dimensional center manifold,  while an operator having strictly negative eigenvalues will correspond to a stable fixed point.  We will use the same general analysis in both cases and will point out the differences in the techniques.

\subsection{Center Manifold Theorem}

We would like to analyze the stability of autonomous quasilinear parabolic equations.  Suppose that $\Phi$ in equation (\ref{quasi}a) is a quasilinear elliptic operator, and suppose that we are in the case where 
\[\sup{\{\Re(\lambda): \lambda \in \sigma(D_{\bar{u}}\Phi) \}} \geq 0.\]
This critical case is complicated, and to treat it, we will work with interpolation spaces.  We will consider the case that 
\[\sigma_+(D_{\bar{u}}\Phi):=\{\lambda \in \sigma(D_{\bar{u}}\Phi): \Re (\lambda) \geq 0 \}\]
consists of a finite number of isolated eigenvalues with finite algebraic multiplicity.

 We would like to study the stability of Ricci Yang-Mills flow at Einstein Yang-Mills metrics using \cite{Si}, which essentially shows that if $\Phi$ is a quasilinear differential operator satisfying certain conditions, with $D\Phi$ having a zero eigenvalue, then the local behavior of the flow near a fixed point is characterized by the presence of a local center manifold.   The theorem that we will use is based upon Theorem 2.2 in \cite{GIK}, which in turn is a compendium of results from Theorems 4.1 and 5.8 and Remark 4.2 in \cite{Si}.  The set-up for the theorem is somewhat complicated, so we first collect the necessary assumptions.

\bigskip

Let $\mathbb{X}_1 \hookrightarrow \mathbb{X}_0$ be a continuous dense inclusion of Banach spaces, and let $\mathbb{X}_\alpha$ and $\mathbb{X}_\beta$ denote the continuous interpolation spaces corresponding to fixed \linebreak
$0<\beta<\alpha<1$.  In other words, $\mathbb{X}_\alpha=(X_0,X_1)_{\alpha}$ and similarly for $\mathbb{X}_\beta$. 
Let 
\begin{equation}
\label{center pde}
\frac{\partial}{\partial t}\vec{x}=A(\vec{x})\vec{x}
\end{equation}
be an autonomous quasilinear parabolic equation posed for $t\geq 0$.  Suppose that $\mathcal{U}_\beta\subset \mathbb{X}_\beta$ is an open set and that
\[A(\cdot) \in C^k(\mathcal{U}_\beta,L(\mathbb{X}_1,\mathbb{X}_0))
\]
for some positive integer $k$.

Additionally, assume that there exists a pair $\mathbb{E}_1 \hookrightarrow \mathbb{E}_0$ of Banach spaces and that there exists an extension $\tilde{A}(\cdot)$ of $A(\cdot)$ to domain $D(\tilde{A}(\cdot))$ that is dense in $\mathbb{E}_0$.  We would like the following statements to hold for all $\vec{x}\in \mathcal{U}_\alpha=\mathcal{U}_\beta \cap \mathbb{X}_\alpha$:
\begin{assumption}
\emph{(Requirements for the Center Manifold Theorem)}
\begin{enumerate}
\item $\mathbb{X}_0 \cong D_{\tilde{A}(\vec{x})}(\theta)\cong (\mathbb{E}_0,D(\tilde{A}(\vec{x})))_\theta$ and $\mathbb{X}_1\cong D_{\tilde{A}(\vec{x})}(1+\theta)\cong (\mathbb{E}_0,D(\tilde{A}(\vec{x})))_{1+\theta}$ for some $\theta \in (0,1)$.  Here $(\cdot,\cdot)$ denotes the continuous interpolation method.  Also $D_{\tilde{A}(\vec{x})}(1+\theta)=\{\vec{x}\in D(\tilde{A}):\tilde{A}\vec{x}\in D_{\tilde{A}(\vec{x})}(\theta)\}$.
\item $\mathbb{E}_1\hookrightarrow \mathbb{X}_\beta \hookrightarrow \mathbb{E}_0$ is a continuous and dense inclusion with the property that there are $C>0$ and $\delta \in (0,1)$ such that for all $\eta \in \mathbb{E}_1$, one has
\[||\eta ||_{\mathbb{X}_\beta}\leq C ||\eta ||_{\mathbb{E}_0}^{1-\delta}||\eta ||_{\mathbb{E}_1}^\delta.\]
\item $\tilde{A}(\vec{x}) \in L(\mathbb{E}_1,\mathbb{E}_0)$ generates a strongly continuous analytic semigroup on $L(\mathbb{E}_0)$;
\item $A(\vec{x})$ agrees with the restriction of $\tilde{A}(\vec{x})$ to the dense subset $D(A) \subseteq \mathbb{X}_0$;
\end{enumerate}
\end{assumption}

Let $\hat{x} \in \mathcal{U}_\alpha$ be a fixed point of equation (\ref{center pde}).  Suppose the spectrum $\sigma$ of the linearized operator $\left. DA\right|_{\hat{x}}$ admits the decomposition $\sigma = \sigma_s \cup \sigma_{cu}$, where $\sigma_s \subset \{z: \Re(z) <0 \}$ and where $\sigma_{cu} \subset \{z: \Re(z) \geq 0 \}$ consists of finitely many eigenvalues of finite multiplicity.  Suppose further that $\sigma_{cu} \subset i \mathbb{R}$, and let $S(\lambda)$ denote the algebraic eigenspace of $\lambda \in \sigma_{cu}$.  In what follows, $B(\mathbb{X},x,d)$ denotes a ball centered at $x$ having radius $d$ measured with respect to $||\cdot||_\mathbb{X}$.

\begin{theorem}
\emph{(Statement of the Center Manifold Theorem)} 
\label{simonett}

$\mathbb{X}_\alpha$ admits the decomposition $\mathbb{X}_\alpha = \mathbb{X}_\alpha^s \oplus \mathbb{X}_\alpha^{cu}$ for all $\alpha \in [0,1]$, where $\mathbb{X}_\alpha^{cu} \equiv \oplus_{\lambda \in \sigma_{cu}} S(\lambda)$.
For each $r \in \mathbb{N}$, there exists $d_r >0$ such that for all $d \in (0,d_r]$, there is a $C^r$ manifold
$\mathcal{M}_{loc}^{cu}$ that is locally invariant for solutions of (\ref{center pde}) as long as they remain in $B(\mathbb{X}_1^{cu},\hat{x},d)\times B(\mathbb{X}_1^s,\hat{x},d)$.
It satisfies $T_{\hat{x}}\mathcal{M}_{loc}^{cu} \cong \mathbb{X}_1^{cu}$, so that $\mathcal{M}_{loc}^{cu}$ is a \emph{local center manifold}.  

For all $\alpha \in (0,1)$, there are constants $C_\alpha >0$ independent of $\hat{x}$ and constants $\omega >0$ and $\hat{d} \in (0,d_0]$ such that one has
\[||\pi^s (\vec{x}(t))-\phi(\pi^{cu}\vec{x}(t))||_{\mathbb{X}_1} \leq \frac{C_\alpha}{t^{1-\alpha}}e^{-\omega t}||\pi^s(\vec{x}(0))-\phi(\pi^{cu}\vec{x}(0))||_{\mathbb{X}_\alpha}
\]
for all solutions $\vec{x}(t)$ with $\vec{x}(0) \in B(\mathbb{X}_\alpha,\hat{x},d)$ and all times $t\geq 0$ such that the solution $\vec{x}(t)$ remains in $B(\mathbb{X}_\alpha,\hat{x},d)$.  Here $\pi^s$ and $\pi^{cu}$ denote the projections onto $\mathbb{X}_\alpha^s \cong (\mathbb{X}_1^s,\mathbb{X}_0^s)_\alpha$ and $\mathbb{X}^{cu}_\alpha$ respectively.
\end{theorem}
To summarize, the presence of eigenvalues with real part being zero corresponds to the existence of an exponentially attractive local center manifold.  The convergence is measured in the $\mathbb{X}_1$ norm for all initial data in an $\mathbb{X}_\alpha$ neighborhood of the fixed point.
\begin{proof}
For the proof of the theorem, we refer the reader to Sections 4 and 5 of \cite{Si}. 
\end{proof}

\subsection{Asymptotic Stability Theorem}
Now suppose that we are in the case of $\sigma_{cu} = \emptyset$; i.e. $\sigma(A)\equiv \sigma_s\subset \{z: \Re(z) <0 \}$.  In this case, our "center manifold" will consist of a single point.  We would still like to use the machinery of Simonett, as this monopolizes upon the smoothing properties of quasilinear parabolic equations and yields an optimal regularity result that solutions in a $\mathbb{X}_\alpha$-neighborhood of a fixed point converge exponentially fast in $\mathbb{X}_1$-norm to the fixed point.  We use the following adaptation of Theorem \ref{simonett} to show that a fixed point is asymptotically stable.

\begin{theorem}
\emph{(Statement of the Asymptotic Stability Theorem)}
\label{simonett 2}

Suppose that the hypotheses of Theorem \ref{simonett} are satisfied.  As before, let $\hat{x}\in \mathcal{U}_{\alpha}$ be a fixed point of equation (\ref{center pde}).  Suppose also that $\sup\{\Re(\lambda): \lambda \in \sigma \}\leq -\delta$ for some $\delta>0$.  Then for all $\alpha \in (0,1)$, there are constants $C_\alpha >0$ independent of $\hat{x}$ and constants $\omega >0$ and $\hat{d}\in (0,d_0]$ such that one has
\[||\vec{x}(t)-\hat{x}||_{\mathbb{X}_1}\leq \frac{C_\alpha}{t^{1-\alpha}}e^{-\omega t}||\vec{x}(0)-\hat{x}||_{\mathbb{X}_\alpha},
\]
for all solutions $\vec{x}(t)$ with $\vec{x}(0)\in B(\mathbb{X}_\alpha,\hat{x},d)$ and all times $t\geq 0$ such that the solution $\vec{x}(t)$ remains in $B(\mathbb{X}_\alpha,\hat{x},d)$.
\end{theorem}
 
\begin{remark}
One can also use the theory of semigroups to prove a slightly less general result.  In \cite{Lu2}, Lunardi shows exponential convergence to a fixed point for quasilinear parabolic PDE for solutions having initial data in a $\mathbb{X}_1$-neighborhood of the fixed point.  The analysis is slightly less complicated than that of Simonett and has the advantage that it can be extended to fully nonlinear PDE, as in Chapter 9 of \cite{Lu}.
\end{remark}

\section{Linearizing the Flow}
We would like to consider the stability of the Ricci Yang-Mills flow at a fixed point.  On a surface, we can let $g=e^uh$, where $h$ is a fixed constant curvature metric.  The Ricci Yang-Mills flow equations then become
\begin{subequations}
\label{rym surface}
\begin{align}
\partial_t u&=\Delta_g u -R_he^{-u}+\frac{1}{2}|F|^2\\
\partial_t a&=-d^\star F.
\end{align}
\end{subequations}
Notice that the equation for $a$ is not quite parabolic; the right-hand side is comprised of ``one-half'' of the laplacian.  We remedy this by using a $1$-parameter family of diffeomorphisms.  
\begin{lemma}
\label{transform}
Equation \ref{rym surface} is equivalent to a parabolic flow via pullback by diffeomorphisms.
\end{lemma}
\begin{proof}
We choose a vector field $W^k=g^{ij}(\Gamma_{ij}^k-\tilde{\Gamma}_{ij}^k)$ for $k=1,2$, where $\tilde{\Gamma}$ is the Christoffel symbol with respect to the fixed background metric $h$.  Notice that if $g_{ij}=e^uh_{ij}$ on a surface, then $W^k=0$.  Additionally, let $W^3=-d^\star a$.   If $h$ is the metric on our principal bundle in the Kaluza-Klein ansatz satisfying equation (\ref{rym surface}), then $\phi_t^\ast h$ satisfies 
\begin{subequations}
\label{grym}
\begin{align}
\partial_t u&=\Delta_g u -R_he^{-u}+\frac{1}{2}|F|^2\\
\partial_t a&=-d^\star F-dd^\star a,
\end{align}
\end{subequations}
where $\phi_t$ is the one-parameter family of diffeomorphisms generated by $W$.  We will call this flow GRYM, and we will choose to work with these equations as they are parabolic.
\end{proof}

We would like to do stability analysis of the fixed points of our flow, which should be the natural geometric limit of the Ricci Yang-Mills flow; namely, we want them to be Einstein Yang-Mills metrics.  An Einstein Yang-Mills metric $(g,a)$ is one such that $g$ is Einstein and $d^\star F_a=0$.    In order to make this work, we must consider a normalized version of equations (\ref{rym surface}) which we will call NGRYM.  We have
\begin{subequations}
\label{ngrym}
\begin{align}
\partial_t u&=\Delta_g u -R_he^{-u}+\frac{1}{2}|F|^2+r-\frac{1}{2}f\\
\partial_t a&=-d^*F-dd^*a,
\end{align}
\end{subequations}
where $r$ and $f$ are the averages of scalar curvature and bundle curvature, respectively; i.e.  $r=\frac{\int_MR dV}{\int dV}$ and $f=\frac{\int_M|F|^2 dV}{\int dV}$.  Since $M$ is a surface, $r$ is a constant in space and time. 
\begin{remark}
NGRYM should be thought of as a certain volume-normalizing flow, in that the volume of the base manifold is fixed.  Due to the lack of scale invariance on the right-hand side of equation (\ref{grym}a), this flow is not quite a rescaling of our original equation.  However, there is evidence that a normalization of this form is useful in proving convergence of the flow \cite{St}.  So our results will be applied to this closely related flow.
\end{remark}

We claim that an Einstein Yang-Mills manifold is a fixed point of the flow.  To see this, we need the following lemma.

\begin{lemma}
\label{Yang-Mills connection}
In the case of a $U(1)$ bundle over a compact surface $M$, a Yang-Mills connection has the property that its curvature is a constant times the volume form; i.e. $F= \lambda dV$, where $\lambda$ is determined by the Chern number of the bundle. 
\end{lemma}
\begin{proof}
We can write $F$ as $F=f(x)dV$ for some function $f$.  $F$ being the curvature of a Yang-Mills connection implies that $d^\star F=0$, and since $U(1)$ is abelian, $dF=0$.  Thus $(dd^\star+d^\star d)F=(dd^\star+d^\star d)(f(x)dV)=0$.  Since $d(dV)=d^\star(dV)=0$, the above shows that $\Delta f=0$.  As $M$ is compact, $f$ must be a constant $\lambda$.

By definition, the Chern number $c$  of a $U(1)$ bundle is given by
\[c =\frac{1}{2\pi}\int_M F.
\]  
But now, we have $c=\frac{1}{2\pi}\int_M \lambda dV$, so $\lambda =\frac{2\pi c}{\int_M dV}$.
\end{proof}

\noindent
Also for an Einstein metric, $u=C$ for some constant $C$. If we write our Yang-Mills connection in the Coulomb gauge ($d^\star a=0$), then we see that an Einstein Yang-Mills metric is a fixed point of the NGRYM. 

One can compute the linearization about an Einstein Yang-Mills metric of the right-hand side of equation (\ref{ngrym}) in the standard fashion.   Let $\left.\frac{\partial u}{\partial t}\right|_{t=0}=v$ and $\left.\frac{\partial a}{\partial t}\right|_{t=0}=b$.   We use the previous characterization of Einstein Yang-Mills metrics, as well as the fact that a Yang-Mills connection is a minimizer of the Yang-Mills functional $\int |F|^2$.  Let $L_1(v,b)$ denote the linearization of equation (\ref{ngrym}a) in the direction of $(v,b)$ and $L_2(v,b)$ denote that of equation (\ref{ngrym}b).
 
\begin{lemma}
The linearization of the right-hand side of equation (\ref{ngrym}) at an Einstein Yang-Mills metric is
\begin{subequations}
\label{linearization}
\begin{align}
L_1(v,b)&=\Delta_h v+(R_h-\lambda^2)v+\lambda\langle db,dV_{h} \rangle,\\
L_2(v,b)&=\Delta_d b -\lambda \mathrm{div} (v dV_h).
\end{align}
\end{subequations}
Here the subscript $h$ denotes quantities measured with respect to the fixed background metric, and $\Delta_d$ is the Hodge-de Rham Laplacian.
\end{lemma}
\begin{proof}
For details, see \cite{Yo}.
\end{proof}

\section{Little H\"{o}lder Spaces}
In order to use the maximal regularity theory, we use spaces that are suitable for this context, namely the little H\"{o}lder spaces.  The little H\"{o}lder space $\mathfrak{h}^{k+\alpha}$ of functions is defined to be the closure of the $C^\infty$ functions with respect to the $||\cdot||_{k+\alpha}$ norm.
\noindent
To be more precise, recall the definition of the H\"{o}lder space $C^\alpha$ of functions:
\[
C^\alpha=\{f\in C_b(I;X):[f]_{C^\alpha(I;X)}:=\sup_{\substack{t,s \in I\\ |t-s|<\delta}}\frac{||f(t)-f(s)||}{|t-s|^\alpha}<\infty \}.
\]
The little H\"{o}lder space of functions are then
\begin{eqnarray*}
\mathfrak{h}^\alpha &=& \{f\in C^\alpha(I;X): \lim_{\delta \to 0}\sup_{\substack{t,s \in I\\ |t-s|<\delta}}\frac{||f(t)-f(s)||}{|t-s|^\alpha}=0 \},\\
\mathfrak{h}^{k+\alpha}&=& \{ f\in C^k_b(I;X): f^{(k)}\in \mathfrak{h}^\alpha(I;X) \}.
\end{eqnarray*}
We can extend this definition to the space of 1-forms on a compact manifold. Let $\mathcal{M}$ be a compact Riemannian manifold.  Fix a background metric $\hat{g}$ and a finite atlas $\{U_{\upsilon
}\}_{1\leq\upsilon\leq\Upsilon}$ of coordinate charts covering $\mathcal{M}$. For each $k\in\mathbb{N}$ and $\alpha\in(0,1]$, let $\mathfrak{h}^{k+\alpha}$ denote the little H\"{o}lder space of 1-forms
with norm $\left\Vert \cdot\right\Vert _{k+\alpha}$ derived from
\[
\left\Vert a\right\Vert _{0+\alpha}:=\max_{\substack{1\leq i\leq
n\\1\leq\upsilon\leq\Upsilon}}\sup_{x,y\in U_{\upsilon}}\frac{\left\vert
a_{i}(x)-a_{i}(y)\right\vert }{(d_{\hat{g}}(x,y))^{\alpha}}.
\]
We state a few facts about these spaces. 

\begin{lemma}
For $j<k$ and $0<\beta<\alpha< 1$, $\mathfrak{h}^{k+\alpha}\hookrightarrow \mathfrak{h}^{j+\beta}$, and this inclusion is continuous and dense.
\end{lemma}

\begin{lemma}
\label{holder interpolation}
For $j\leq k \in \mathbb{N}$, $0<\beta <\alpha<1$, and $0<\theta<1$, if
\[\theta(k+\alpha)+(1-\theta)(j+\beta)\]
is not an integer, then there is a Banach space isomorphism
\begin{equation}
(\mathfrak{h}^{j+\beta},\mathfrak{h}^{k+\alpha})_\theta\cong \mathfrak{h}^{(\theta k +(1-\theta)j)+(\theta \alpha +(1-\theta)\beta)},
\end{equation}
and there exists $C<\infty$ such that for all $\eta \in \mathfrak{h}^{k+\alpha}$
\begin{equation}
||\eta ||_{(\mathfrak{h}^{j+\beta},\mathfrak{h}^{k+\alpha})_\theta}\leq C ||\eta||_{\mathfrak{h}_{j+\beta}}^{1-\theta}||\eta||_{\mathfrak{h}_{k+\alpha}}^{\theta}.
\end{equation}
Namely, these spaces form a continuous interpolation method.
\end{lemma}

Let $D_A(\theta)$ be the continuous interpolation spaces.  We can show that for certain choices of $D(A)$ and $\mathbb{X}$ these in fact are the little H\"{o}lder spaces.  Suppose $A$ is given by
\[\left\{\begin{array}{c}D(A)=\{u\in \cap_{p\geq 1}W^{2,p}_{loc}: u, Au \in C(\mathbb{R}^n)\}\\A:D(A)\mapsto C(\mathbb{R}^n)\end{array} \right.
\]
\begin{proposition}
Let $0<\theta<1$.  Then
\[D_A(\theta)=\mathfrak{h}^{2\theta}, \mbox{  if  } \theta \neq \frac{1}{2}.
\]
\end{proposition}
\begin{proof}
See the proof of Theorem 3.1.12 in \cite{Lu}.  
\end{proof}

\section{Application of the Center Manifold Theorem to the NGRYM}
We would like to show that we can apply Theorems \ref{simonett} and \ref{simonett 2} to NGRYM.  In order to put our analysis into this framework, we need to define appropriate spaces.

Fix $0<\delta<\epsilon<1$.  We would like the consider the following hierarchy of Banach spaces 
\[\mathbb{X}_1\subset \mathbb{E}_1 \subset \mathbb{X}_0 \subset \mathbb{E}_0,
\]
where 
\begin{equation*}
\begin{array}[c]{l}
\mathbb{X}_1 = \{(u,a): u \in \mathfrak{h}^{2+\epsilon}, a \in \mathfrak{h}^{2+\epsilon} \},\\
\mathbb{E}_1 = \{(u,a): u \in \mathfrak{h}^{2+\delta}, a \in \mathfrak{h}^{2+\delta} \},\\
\mathbb{X}_0= \{(u,a): u \in \mathfrak{h}^{\epsilon}, a \in \mathfrak{h}^{\epsilon} \},\\
\mathbb{E}_0=\{ (u,a): u \in \mathfrak{h}^{\delta}, a \in \mathfrak{h}^{\delta} \}.
\end{array}
\end{equation*}
Notice that in the notation above, we mean $(u,a)\in \mathfrak{h}^{\cdot}(M)\times \mathfrak{h}^{\cdot}(\Omega^1(M))$.  We would like to check that Assumption (1) holds.  That is, we want to show that $X_0=(E_0, E_1)_\theta$ and $X_1=(E_0, E_1)_{1+\theta}$.  Clearly the product of interpolation spaces is also an interpolation space.  Notice that if $\theta=\frac{\epsilon-\delta}{2}$, then $\mathfrak{h}^{2+\epsilon}=(\mathfrak{h}^{2+\delta},\mathfrak{h}^{\delta})_{1+\theta}$ and $\mathfrak{h}^\epsilon=(\mathfrak{h}^{2+\delta},\mathfrak{h}^{\delta})_\theta$ by Lemma \ref{holder interpolation}.  

Let $\mathbb{X}_\beta=(\mathbb{X}_0,\mathbb{X}_1)_\beta$ and $\mathbb{X}_\alpha=(\mathbb{X}_0,\mathbb{X}_1)_\alpha$ for fixed $\frac{1}{2} \leq \beta <\alpha <1$.  We would like to check that $\mathbb{E}_1 \hookrightarrow \mathbb{X}_\beta \hookrightarrow \mathbb{E}_0$ is a continuous and dense inclusion and that the interpolation inequality holds.  This will fulfill Assumption (2) of Theorem \ref{simonett}.  Notice that $X_0$ and $X_1$ can be gotten from interpolating between $E_0$ and $E_1$.  If $\theta = \frac{\epsilon-\delta}{2}$, then by lemma \ref{holder interpolation},
\begin{equation}
\label{X_0 interp}
\mathbb{X}_0=(\mathbb{E}_0,\mathbb{E}_1)_\theta,
\end{equation}
\begin{equation}
\label{X_1 interp}
\mathbb{X}_1=(\mathbb{E}_0,\mathbb{E}_1)_{1+\theta}.
\end{equation}  

\begin{lemma}
There exists $C>0$, $\rho \in (0,1)$, and  $\beta \in (\frac{1}{2},1)$ such that $\mathbb{E}_1 \hookrightarrow \mathbb{X}_\beta \hookrightarrow \mathbb{E}_0$ and 
\[||\eta||_{\mathbb{X}_\beta}\leq C ||\eta||_{\mathbb{E}_0}^{1-\rho}
||\eta||_{\mathbb{E}_1}^{\rho},\]
for all $\eta \in \mathbb{X}_\beta$.
\end{lemma}
\begin{proof}
Since $\beta \in (\frac{1}{2},1)$, it is clear that the inclusions hold.  Using equations (\ref{X_0 interp}) and (\ref{X_1 interp}),we have
\begin{eqnarray*}
||\eta||_{\mathbb{X}_\beta}&\leq& C ||\eta||_{\mathbb{X}_0}^{1-\beta}||\eta||_{\mathbb{X}_1}^{\beta}\\
&\leq&C(||\eta||_{\mathbb{E}_0}^{1-\theta}||\eta||_{\mathbb{E}_1}^{\theta})^{1-\beta}(||\eta||_{\mathbb{E}_0}^{1-(1+\theta)}||\eta||_{\mathbb{E}_1}^{1+\theta})^{\beta}\\
&\leq&C||\eta||_{\mathbb{E}_0}^{1-\rho}||\eta||_{\mathbb{E}_1}^{\rho},
\end{eqnarray*} 
where $\rho =\theta+\beta$.  It remains only to check that $\rho \in (0,1)$.  Since $0<\delta<\epsilon<1$, $0<\theta<\frac{1}{2}$, we simply choose $\beta$ to be in $(\frac{1}{2},1-\theta)$.
\end{proof}

Since we would like to use Theorem \ref{simonett}, we want to make our notation match that of Simonett.  In other words, we would like to write equation (\ref{ngrym}) as 
\[\partial_t \vec{x}=A(\vec{x})\vec{x}.
\]
In a fixed coordinate system, we can write the right-hand side of equation (\ref{ngrym}) as
\begin{equation}
\label{A}
A(u,a)(u,a)=\left( \begin{array}{c}a(x,g)^{ij}\partial_i\partial_j u+b(x,g,\partial g)^k\partial_k u +c(x,g,da)u \\d(x,g)^{ij}\partial_i\partial_j a_k+e(x,g,\partial g)^i\partial_i a_k-f(x)a_k\end{array} \right), 
\end{equation}
where the functions $a(x,\cdot), b(x,\cdot, \cdot), c(x,\cdot,\cdot), d(x,\cdot), e(x,\cdot,\cdot), f(x)$ depend smoothly on $x\in M$.  They are analytic functions of their remaining arguments.  

We want to show that for all $(u,a)$ in a certain open set, $A(u,a)$ is a bounded map from $\mathbb{X}_1$ to $\mathbb{X}_0$.  Since $\mathbb{X}_1$ is a dense subspace of $\mathbb{E}_1$, we have an extension operator $\tilde{A}$.  We will show also that $\tilde{A}:\mathbb{E}_1\to\mathbb{E}_0$ is a bounded operator.  For a fixed $r>0$, we define the open subsets $\mathcal{U}_\beta \subset \mathbb{X}_\beta$ and $\mathcal{U}_\alpha \subset \mathbb{X}_\alpha$ to be 
\[\mathcal{U}_\beta := \{(u,a):||u||_{\mathbb{X}_\beta}>r, ||a||_{\mathcal{X}_\beta}>r \},
\]
\[\mathcal{U}_\alpha:= \mathcal{U}_\beta \cap \mathbb{X}_\alpha.
\]

\begin{lemma}
\label{bounded operator}
For $(u,a)\in \mathcal{U}_\beta$, $A(u,a)\in L(\mathbb{X}_1,\mathbb{X}_0)$.  Also for $(u,a)\in \mathcal{U}_\alpha$, $\tilde{A}(u,a)\in L(\mathbb{E}_1,\mathbb{E}_0)$.
\end{lemma}
\begin{proof}
The proof follows in almost exactly the same way as the proof of Lemma 3.3 in \cite{GIK}.  \end{proof}

Let $X$ be a Banach space.  A semigroup $S(t)\subset L(X)$ is said to be analytic if $t\mapsto S(t)$ is an analytic map for all $t\in (0,\infty)$.  Additionally, $S(t)$ is strongly continuous if $t\mapsto S(t)x$ is continuous on $[0,\infty)$ for all $x\in X$.  In order to satisfy Assumption (3) and to show that $\tilde{A}$ generates a strongly continuous analytic semigroup on $L(\mathbb{E}_0)$, we need the following lemma and definition from \cite{Lu}.
\begin{lemma}
\label{semigroup}
$A:D(A)\subset X \to X$ generates a strongly continuous analytic semigroup if $A$ is sectorial and the domain $D(A)$ is dense in $X$.
\end{lemma}
\begin{definition}
\label{sectorial def}
$A$ is \emph{sectorial} if there exist constants $\omega \in \mathbb{R}, \theta \in (\frac{\pi}{2},\pi)$,and $M>0$ such that 
\begin{description}
\item[i.  ]$\rho(A) \supset S_{\theta,\omega}=\{\lambda \in \mathbb{C}:\lambda\neq \omega, |\arg(\lambda-\omega)|<\theta \},$
\item[ii.  ]$||R(\lambda,A)||_{L(X)}\leq\frac{M}{|\lambda-\omega|}, \forall \lambda \in S_{\theta,\omega}$.
\end{description}
\end{definition}

We see that we need to show that for $(u,a) \in \mathcal{U}_\alpha$, $\tilde{A}(u,a)$ is a sectorial operator in $\mathbb{E}_0$.  We would like to use the following Proposition from \cite{Lu}.
\begin{proposition}
\label{sectorial piece}
Let $A: D(A)\subset X\mapsto X$ be a linear operator, and let $\alpha \in (0,1)$.  Define $A_\alpha:D_A(\alpha+1)\mapsto D_A(\alpha)$ by $A_\alpha x:=Ax$.  In other words, $A_\alpha$ is the piece of $A$ defined on $D_A(\alpha)$.  Then $A_\alpha$ is a sectorial operator in $D_A(\alpha)$.
\end{proposition}

\begin{lemma}
\label{A sectorial}
For $(u,a)\in \mathcal{U}_\alpha$, $\tilde{A}(u,a):\mathbb{E}_1 \to \mathbb{E}_0$ is a sectorial operator on $\mathbb{E}_0$.
\end{lemma}
\begin{proof}
In our case, there exists $\theta \in (0,1)$ such that $\mathbb{X}_0=D_{\tilde{A}}(\theta)$ and $\mathbb{X}_1=D_{\tilde{A}}(1+\theta)$.  Thus, by Proposition \ref{sectorial piece}, $A(u,a):\mathbb{X}_1\mapsto\mathbb{X}_0$ is a sectorial operator on $\mathbb{X}_0$. 

Now let $\eta \in \mathbb{E}_1$, and since $\mathbb{X}_1$ is densely embedded in $\mathbb{E}_1$, let $\{\eta_i\}\in \mathbb{X}_1$ such that $\eta_i\to \eta$ in $\mathbb{E}_1$.  Using the notation in Definition \ref{sectorial def}, fix $\lambda \in S_{\theta,\omega}$.  Then we have
\begin{align*}
||R(\lambda,A)\eta_i||_{\mathbb{E}_0}
&\leq c||R(\lambda,A)\eta_i||_{\mathbb{X}_0}\\
&\leq \frac{M}{|\gamma-\omega|}||\eta_i||_{\mathbb{X}_1}.
\end{align*}
Since the resolvent operator is continuous, and $A=\tilde{A}$ on the dense set $\mathbb{X}_1$, we can pass to the limit to obtain that $\tilde{A}(u,a)$ is sectorial in $\mathbb{E}_0$.
\end{proof}

\begin{lemma}
For $(u,a)\in \mathcal{U}_\alpha$, $\tilde{A}(u,a):\mathbb{E}_1 \to \mathbb{E}_0$ generates an analytic semigroup on $L(\mathbb{E}_0)$.
\end{lemma}
\begin{proof}
This follows from Lemmas \ref{semigroup} and \ref{A sectorial}
\end{proof}

\subsection{$\lambda=0$}
We would like to compute the spectrum of $L$ at an Einstein Yang-Mills metric.  Since $L$ is self-adjoint, we know that the spectrum is pure point and that it is contained in $\mathbb{R}$.  We first consider the case of a flat bundle over a constant curvature Riemann surface ($\lambda=0$ in the equations above).  We consider the linearization of equation (\ref{ngrym}) to be the following operator: 
\begin{equation}
\label{ngrym flat}
L\left(\begin{array}{c}v\\ b \end{array} \right)=\left( \begin{array}{cc}\Delta +R_h&0\\0&\Delta_d \end{array} \right) \left(\begin{array}{c}v\\ b \end{array} \right).
\end{equation}
First, we would like to compute the spectrum of $L$.
\begin{lemma}
\label{spectrum flat 1}
If $R_h\leq 0$, then the $L^2$ spectrum of $L$ is contained in $(-\infty,0]$.  In particular, $L$ is linearly stable.
\end{lemma}
\begin{proof}
We use the natural $L^2$ inner product for product spaces
\[(\left(\begin{array}{c}v\\ b \end{array} \right),\left(\begin{array}{c}w\\ \rho \end{array} \right))_{L^2}=\int((v,w)+(b,\rho))d\mu.
\]
Then, in the case of $R_h\leq 0$, we have
\begin{eqnarray*}
(L\left(\begin{array}{c}v\\ b \end{array} \right),\left(\begin{array}{c}v\\ b \end{array} \right))_{L^2}&=&\int ((\Delta v,v)+R_h(v,v)+(\Delta_d b,b)) d\mu\\
&=&-||\nabla v||^2+R_h||v||^2-||d b||^2\\
&\leq&0.
\end{eqnarray*}
We can notice that in the case of $R_h<0$, the zero eigenvalue of $L$ corresponds to $v=0$ and $b$ harmonic.  Thus our center manifold is $2g$-dimensional, where $g$ is the genus of $M$.  For $R_h=0$, the zero eigenvalue corresponds to $v$ and $b$ harmonic.  So this center manifold will have dimension $1+2=3$.
\end{proof}

\begin{lemma}
\label{spectrum flat 2}
If $R_{h} >0$, then $\sigma(L)\cap (0,\infty)\neq \emptyset$.  In particular, $L$ is not linearly stable.
\end{lemma}
\begin{proof}
Since there are no harmonic 1-forms over the 2-sphere, it is clear that there are no unstable directions corresponding to pairs of the form $(0,b)$.  Consider pairs of the form $(v,0)$.  Suppose there exists an eigenvalue $\gamma$ for some $(v_0, 0)$.  In this case, we would have
\begin{equation*}
\Delta v_0=(\gamma-2)v_0=\mu v_0.
\end{equation*} 
The eigenvalues of the laplacian over the n-sphere (having radius 1) are given by $\mu_k=-k(k+n-1)$.  In our case, we have $\mu_0=0, \mu_1=-2, \ldots$, so the above equation is equivalent to
\begin{equation*}
\gamma-2=\mu_0,
\end{equation*}
i.e. $\gamma=2$.  Thus we have a positive eigenvalue for $L$ corresponding to the first eigenvalue of the laplacian.  This unstable direction is given by the 1-dimensional space of constant functions.  Additionally, we see that $\gamma=0$ can be obtained by the second eigenvalue.  This corresponds to the 2-dimensional space of homogeneous hermitian polynomials of degree 1. \cite{Cha}
\end{proof}

\begin{remark}
As Knopf has observed in \cite{Kn}, if one is concerned only with instabilities that are geometrically meaningful, then, up to diffeomorphims,  one needs only deal with perturbations that preserve volume.   Along such perturbations, we see that $L$ is linearly stable.  
\end{remark}

\subsection{$\lambda \neq 0$}

Now we consider the case where our bundle is not flat.  In this case, our operator has the form
\begin{equation}
\label{L nonflat}
L \left(\begin{array}{c}v\\ b \end{array} \right) =\left(\begin{array}{cc}\Delta v+(R_h-\lambda^2)v&\lambda \langle db,dV \rangle\\-\lambda \operatorname*{div}(v dV)&\Delta_d b \end{array} \right).
\end{equation}

\begin{lemma}
$L:L^2(C^\infty(M))\oplus L^2(\Omega^1(M))\to L^2(C^\infty(M))\oplus L^2(\Omega^1(M))$ is a self-adjoint operator.
\end{lemma}
\begin{proof}
By definition of the $L^2$ inner product, we have
\begin{align*}
(L\left(\begin{array}{c}v\\ b \end{array} \right),\left(\begin{array}{c}w\\ c \end{array} \right))_{L^2}= & \int ((\Delta v,w)+(-\lambda^2+R_h)(v,w)+(\Delta_d b,c)) d\mu\\
		&+\int\lambda(db,dV)w+(-\lambda \operatorname*{div}(v dV),c)d\mu.
\end{align*}
Clearly the first integral is equal to
$\int ((v,\Delta w)+(R_{h}-\lambda^2)(v,w)+(b,\Delta_d c))d\mu$.  The second integral becomes
\begin{eqnarray*}
\int(\lambda(db,dV)w-\lambda  \operatorname*{div}(v dV),c)d\mu
&=&\int(\lambda(db,wdV)+\lambda(d^\star(vdV),c))d\mu\\
&=&\int(\lambda(b,d^\star(wdV)+\lambda(vdV,dc))d\mu\\
&=&\int(-\lambda(b,\operatorname*{div}(w dV))+\lambda(dc,dV)v)d\mu.
\end{eqnarray*}
\noindent
Thus we see that indeed $(L\left(\begin{array}{c}v\\ b \end{array} \right),\left(\begin{array}{c}w\\ c \end{array} \right))_{L^2}=(\left(\begin{array}{c}v\\ b \end{array} \right),L\left(\begin{array}{c}w\\ c \end{array} \right))_{L^2}$; i.e. $L$ is self-adjoint.
\end{proof}

\begin{lemma}
If $R_h \leq 0$, then $\sigma(L) \subset (-\infty,0]$.
\end{lemma}
\begin{proof}
We have the following simple computation:
\begin{eqnarray*}
(L\left(\begin{array}{c}v\\ b \end{array} \right),\left(\begin{array}{c}v\\ b \end{array} \right))
&=&(\Delta v,v)+(R_{h}-\lambda^2)||v||^2+(\Delta_d b,b)+2 \lambda(db,vdV)\\
&=&-||\nabla v||^2+(R_{h}-\lambda^2)||v||^2-||db||^2-||d^\star b||^2\\
& &+||db||^2+\lambda^2||v||^2\\
&\leq&0,
\end{eqnarray*}
\noindent
where the second line follows from Cauchy-Schwartz.

We claim that $\sigma(L)$ contains zero eigenvalues corresponding to ordered pairs of the form $(0,b)$, where $b$ is harmonic.  In particular, $b$ being harmonic implies that $db=0$, so the computation above implies that such pairs yield a zero eigenvalue.  These are, in fact, the only zero eigenvalues.  In general, we have the following estimate:
\begin{eqnarray*}
(L\left(\begin{array}{c}v\\ b \end{array} \right),\left(\begin{array}{c}v\\ b \end{array} \right))&\leq&
-||\nabla v||^2+(R_{h}-\lambda^2)||v||^2-||db||^2-||d^\star b||^2\\
& &+||db||^2+\lambda^2||v||^2\\
&=&-||\nabla v||^2+R_{h}||v||^2-||d^\star b||^2.
\end{eqnarray*}
As long as $v$ is not a constant or $d^\star b \neq 0$, then $(L\left(\begin{array}{c}v\\ b \end{array} \right),\left(\begin{array}{c}v\\ b \end{array} \right))<0$.  If both $v=C$ and $d^*b=0$ (and $db \neq 0$), then 
\begin{eqnarray*}
L\left(\begin{array}{c}v\\ b \end{array} \right)&=&\left(\begin{array}{cc}(R_h-\lambda^2)v&0\\ 0&db \end{array} \right)\\
&\neq& 0 \left(\begin{array}{c}v\\ b \end{array} \right).
\end{eqnarray*}
If $v=C$ and $b$ is harmonic, then clearly there is no zero eigenvalue.
\end{proof}

\begin{lemma}
If $R_h>0$, the spectrum of $L$ can be computed in several cases.  If $|\lambda|=\frac{1}{2}$ or $1$, then $\sigma(L)\cap (0,\infty)\neq \emptyset$.  If $|\lambda| >1$, then $\sigma(L)\subset (-\infty,-\delta]$ for some $\delta>0$.
\end{lemma}
\begin{proof}
By Lemma \ref{Yang-Mills connection}, we see that $\lambda$ can only attain specific values determined by the Chern number of the bundle.  The $U(1)$-bundles over $S^2$ are in 1-1 correspondence with elements of $\mathbb{Z}$ and are determined up to equivalence by their $1^{st}$ Chern class.    In particular, two bundles are equivalent if they have $n\in \mathbb{Z}$ as their Chern number (see, for example, Chapter 6.1 in \cite{Na}).   Then $\lambda =\frac{n}{2}$, again by Lemma \ref{Yang-Mills connection}.  We first consider the case of $|\lambda| >1$.  As above, we can compute
\begin{align*}
(L\left(\begin{array}{c}v\\ b \end{array} \right),\left(\begin{array}{c}v\\ b \end{array} \right))
&=(\Delta v,v)+(2-\lambda^2)||v||^2+(\Delta_d b,b)+2 \lambda(db,vdV)\\
&=-||\nabla v||^2+(2-\lambda^2)||v||^2+\alpha 2\lambda\int (db,vdV)\\
& +(1-\alpha)2\lambda\int (db,vdV)+\alpha \int(\Delta_db,b)+(1-\alpha)\int(\Delta_db,b)\\
&\leq -||\nabla v||^2+(2-\lambda^2)||v||^2+\alpha(\frac{1}{\alpha}||\nabla v||^2+\frac{\alpha \lambda^2}{4}||b||^2)\\
& +(1-\alpha)(||db||^2+\lambda^2||v||^2)-2\alpha||b||^2\\
& +(1-\alpha)(-||db||^2-||d^\star b||^2)\\
&\leq (2-\alpha \lambda^2)||v||^2+\alpha(\frac{\alpha \lambda^2}{4}-2)||b||^2,
\end{align*}
where $\alpha \in (0,1)$ is to be chosen later.  We also used Cauchy-Schwartz and the fact that the first eigenvalue of $\Delta_d$ acting on 1-forms on $S^2$ is $-2$.  Then we need to find $\delta >0$ and $\alpha \in (0,1)$ such that $2-\alpha \lambda^2\leq -\delta$ and $\alpha(\frac{\alpha \lambda^2}{4}-2)\leq -\delta$.  This amounts to the bounds $\frac{2+\delta}{\alpha}\leq \lambda^2 \leq \frac{8}{\alpha}-\frac{4\delta}{\alpha^2}$.  It is clear that for $\lambda^2 >2$, we can choose such an $\alpha$ and a $\delta$ small enough to make these bounds hold.  Since $\lambda$ is quantized, the smallest such value we have to apply these bounds to is $\lambda^2=\frac{9}{4}$.  In these cases, the spectrum is strictly negative.

Let us now consider the remaining two cases:  $|\lambda| =\frac{1}{2}$ and $|\lambda| = 1$.  In both of these cases, we can explicitly show the existence of a positive eigenvalue, corresponding to ordered pairs of the form $(v,0)$, where $v$ is constant.  The computation is the same as that in the $\lambda=0$ case;  we obtain $\mu=\frac{7}{4}$ for $\lambda=\frac{1}{2}$ and $\mu=1$ for $\lambda=1$ as eigenvalues for $L_\lambda$.
\end{proof}

In the case of $R_h\leq 0$, we have zero eigenvalues, so we again have the existence of a center manifold and can apply Theorem \ref{simonett}.  For $R_h>0$ and $|\lambda| = \frac{1}{2}, 1$, we again have unstable directions corresponding to the constant functions on $S^2$.  For $R_h>0$ and $|\lambda| >1$, we will be able to apply Theorem \ref{simonett 2}.

\section{Stability of the Ricci Yang-Mills Flow}
We are finally ready to state the center manifold theorem for the Ricci Yang-Mills flow over surfaces with $R_h\leq 0$.
\begin{theorem}
\label{center manifold ngrym}
Let $(u,a)$ be a Yang-Mills connection over a surface with constant curvature $R\leq 0$.  Then $\mathbb{X}_\alpha$ admits the decomposition 
\[\mathbb{X}_\alpha^s \oplus \mathbb{X}_\alpha^c,
\]
where $\mathbb{X}_\alpha^c$ corresponds to the algebraic eigenspace of 0.  There exists $d_0 >0$ such that for all $d \in (0,d_0]$, there is a $C^\infty$ manifold
$\mathcal{M}_{loc}^{c}$ that is locally invariant for solutions of (\ref{ngrym}) as long as they remain in $B(\mathbb{X}_1^{c},(u,a),d)$.  It is
such that $T_{(u,a)}\mathcal{M}^c_{loc}\cong \mathbb{X}^c$.  $\mathcal{M}^c_{loc}$ is a unique local center manifold consisting of Einstein Yang-Mills metrics. $\mathcal{M}_{loc}^c$ is 2g-dimensional for $R_h<0$ and 3-dimensional for $R_h=0$.

\noindent
There are constants $C>0$, $\omega>0$, and $d\in (0,d_0]$, such that
\[||\pi^s ((v,b))-\phi(\pi^{c}((v,b))||_{\mathbb{X}_1}\leq Ce^{-\omega t}||\pi^s((v(0),b(0))-\phi(\pi^c((v(0),b(0)))||_{\mathbb{X}_\alpha}
\]
for all solutions $(v,b)$ with $(v(0),b(0)) \in B(\mathbb{X}_\alpha, (u,a),d)$ and all times $t\geq 0$ such that the solutions remain in this ball.  Here $\pi^s$ and $\pi^c$ denote the projections onto $\mathbb{X}_\alpha^s$ and $\mathbb{X}_\alpha^c$ respectively.  
\end{theorem}
\begin{remark}
In particular, this theorem states that any bundle that solves NGRYM with $(v(0),b(0))$ close enough to an Einstein Yang-Mills metric will have its conformal factor and connection 1-form converge exponentially fast to those of the Einstein Yang-Mills metric.  
\end{remark}
\begin{proof}
By Theorem \ref{simonett}, we obtain the existence of local $C^r$ center manifolds to which solutions to NGRYM that are sufficiently close to $(u,a)$ converge exponentially fast, as long as solutions remain in the given neighborhood of the fixed point.  Notice that the family of center manifolds $\mathcal{M}_{loc}^c(r)$  are in fact independent of $r$ and consist precisely of Einstein Yang-Mills connections.  To see why this is so, let $(v,b)$ be an Einstein Yang-Mills connection sufficiently close to $(u,a)$ that is not contained in $\mathcal{M}_{loc}^c$.  By Theorem \ref{simonett}, $(v,b)$ would converge exponentially fast to the center manifold.  But this contradicts the fact that $(v,b)$ is a fixed point.  Since the space of Yang-Mills connections over a Riemann surface is 2g-dimensional, we see that the center manifolds consist precisely of such pairs.  In the case of $R_h=0$, the local center manifolds again consist of Einstein Yang-Mills metrics, but we allow the conformal factor to be any constant.  So the dimension is 3.  The analysis follows in the same way.

Finally, we would like to check that solutions to NGRYM that start in a sufficiently small neighborhood of an Einstein Yang-Mills metric actually stay there.  Notice that $|\frac{\partial}{\partial t}v|=|R_{e^vh}+r+\frac{1}{2}|F|^2-\frac{1}{2}f|\leq C_1e^{-\omega_1t}$ for some $C_1,\omega_1>0$ as long as $(v,b)$ stays in $B(\mathbb{X}_\alpha,(u,a),d)$.  Also, $|\frac{\partial}{\partial t}b|=|-d^\star F-dd^\star b|\leq C_2e^{-\omega_2t}$ for some $C_2,\omega_2>0$ while $(v,b)$ stays in the ball.  Let $0<d^\prime<d$ small such that for all $(\bar{v},\bar{b})$ with initial data $(\bar{v},\bar{b})(0)\in B(\mathbb{X}_\alpha,(u,a),d^\prime)$,
\[|\bar{v}(t)-u|\leq |\bar{v}(t)-\bar{v}(0)|+|\bar{v}(0)-u|<d,
\]
and similarly for $\bar{b}$.  These estimates are independent of time, so we see that $(v,b)$ remains in $B(\mathbb{X}_\alpha,(u,a),d)$.
\noindent
The rest of the theorem follows from Theorem \ref{simonett}.  
\end{proof}

Now we consider the case of $R_h>0$.  In this setting, our stability result depends on the value of $\lambda$.  For $|\lambda| \geq \frac{3}{2}$, we saw that there exists a $\delta >0$, depending on $\lambda$ such that $ \sigma(L) \subset (-\infty,-\delta]$.  So we obtain the following theorem.

\begin{theorem}
\label{asymp stability of ngrym}
Let $(u,a)$ be an Einstein Yang-Mills metric over a surface of constant curvature $R_h>0$ with Chern number $|c| \geq 3$  and let $\delta_0 \in [0,\delta)$.  Then for all $\alpha \in (0,1)$, there are constants $C_\alpha$ independent of $(u,a)$ and $\hat{d}\in (0,d_0]$ such that, if $(\bar{u},\bar{a})(0)\in B(\mathbb{X}_\alpha,\hat{d},(u,a))$, then
\[||(\bar{u},\bar{a})(t)-(u,a)||_{\mathbb{X}_1}\leq \frac{C_\alpha}{t^{1-\alpha}}e^{-{\delta_0} t}||(\bar{u},\bar{a})(0)-(u,a)||_{\mathbb{X}_\alpha}.
\]
as long as $(\bar{u},\bar{a})(t)$ stays in $B(\mathbb{X}_\alpha,\hat{d},(u,a))$.
\end{theorem}
\begin{proof}
We begin by noting that $(u,a)$ is a unique fixed point, since we have fixed a gauge.  Then the proof of the theorem follows in the same fashion as that of Theorem \ref{center manifold ngrym}. 
\end{proof}

We would like to use a lemma from \cite{GIK} to show that the convergence of NGRYM implies that of NRYM.  

\begin{lemma}
\emph{(Lemma 3.5, \cite{GIK})}
Let $V(t)$ be a vector field on a Riemannian manifold $(M^n,g(t))$, where $0\leq t<\infty$, and suppose there are constants $0<c\leq C<\infty$ such that 
\[\sup_{x\in M^N}|V(x,t)|_{g(t)}\leq Ce^{-ct}.
\]
Then the diffeomorphisms $\phi_t$ generated by $V$ converge exponentially to a fixed diffeomorphism $\phi_\infty$ of $M$.
\end{lemma}

\begin{proposition}
Let $(u_0,a_0)$ be an Einstein Yang-Mills metric with $a_0$ written in the Coulomb gauge.  Suppose there exists a neighborhood $\mathcal{O}$ of $(u_0,a_0)$ measured in the $||\cdot||_{2\alpha+\epsilon}$ norm such that for every $(\tilde{u}_0,\tilde{a}_0)\in \mathcal{O}$, the unique solution $(\bar{u},\bar{a})$ to NGRYM converges to an Einstein Yang-Mills metric $(\bar{u}_\infty,\bar{a}_\infty)$.  Then the unique solution $(\tilde{u},\tilde{a})$ to NRYM with initial data $(\tilde{u}_0,\tilde{a}_0)$ converges exponentially fast to an Einstein Yang-Mills metric $(\tilde{u}_\infty,\tilde{a}_\infty)$.
\end{proposition}
\begin{proof}
Since $F$ is invariant under gauge transformation, it is clear that $\tilde{a}_\infty$ is Yang-Mills.  So we need to show that $\tilde{a}$ converges to a limit.  We have that $\bar{a}\to \bar{a}_\infty$ exponentially fast, so in particular, $d^\star \bar{a}\to 0$ exponentially fast.  Thus our vector field $W$ from Lemma \ref{transform} converges to 0 exponentially fast.  Our result follows from the previous lemma. 
\end{proof}

\end{document}